\documentclass[10pt, a4paper,reqno]{amsart}
\usepackage{amsthm,euscript,colonequals}
\usepackage{mathtools} 
\usepackage{stmaryrd}
\theoremstyle{}
{\theoremstyle{definition}
\newtheorem{dfn}{Definition}[section]}
\newtheorem{prop}[dfn]{Proposition}

\newtheorem{thm}[dfn]{Theorem}
{\theoremstyle{definition}
\newtheorem{rem}[dfn]{Remark}}
\newtheorem{lem}[dfn]{Lemma}
\newtheorem{cor}[dfn]{Corollary}

{\theoremstyle{definition}
\newtheorem{exa}[dfn]{Example}}

\usepackage[top=30truemm,bottom=30truemm,left=35truemm,right=35truemm]{geometry}
\usepackage{amsmath,amssymb,amscd,bm,graphicx,tikz-cd,upgreek,dsfont,amsfonts,tikz}
\usepackage[all]{xy}
\usepackage[colorlinks]{hyperref}

\tikzset{
        DB/.style={circle,draw=black,circle,fill=white,inner sep=0pt, minimum size=5pt},
        DW/.style={circle,draw=black,fill=black,inner sep=0pt, minimum size=5pt},
         gap/.style={inner sep=0.5pt,fill=white},
}
\tikzset{
W/.style={circle,draw=black,circle,fill=white,inner sep=0pt, minimum size=4pt},
B/.style={circle,draw=black!80!white,circle,fill=black!80!white,inner sep=0pt, outer sep=4pt, minimum size=3pt},
Bs/.style={circle,draw=black!80!white,circle,fill=black!80!white,inner sep=0pt, outer sep=2pt, minimum size=3pt},
BL/.style={circle,draw=blue!60!white,circle,fill=blue!60!white,inner sep=0pt, minimum size=4pt},
R/.style={circle,draw=red!60!white,circle,fill=red!60!white,inner sep=0pt, minimum size=4pt},  
G/.style={circle,draw=green!65!black,circle,fill=green!65!black,inner sep=0pt, minimum size=4pt},     
Rs/.style={circle,draw=red!60!white,circle,fill=red!60!white,inner sep=0pt, minimum size=2pt}, 
BLs/.style={circle,draw=blue!60!white,circle,fill=blue!60!white,inner sep=0pt, minimum size=2pt},
Gs/.style={circle,draw=green!65!black,circle,fill=green!65!black,inner sep=0pt, minimum size=2pt},  }

\usepackage{setspace}
\newcommand{\marginparstretch}{0.6}
\let\oldmarginpar\marginpar
\renewcommand\marginpar[1]{\-\oldmarginpar[\framebox{\setstretch{\marginparstretch}\begin{minipage}{\marginparwidth}{\raggedleft\tiny #1}\end{minipage}}]{\framebox{\setstretch{\marginparstretch}\begin{minipage}{\marginparwidth}{\raggedright\tiny #1}\end{minipage}}}}

\usepackage{enumitem}
\setlist[enumerate]{format=\normalfont}

\usepackage{array,multirow,booktabs,longtable}
\numberwithin{equation}{section}

\newcommand{\ii}{\kern 1pt {\rm i}\kern 1pt }

\newcommand{\Spec}{\operatorname{Spec}}

\newcommand{\Hom}{\operatorname{Hom}}
\newcommand{\Ext}{\operatorname{Ext}}

\newcommand{\RHom}{\operatorname{{\mathbf R}Hom}}
\newcommand{\End}{\operatorname{End}}
\newcommand{\uEnd}{\operatorname{\underline{End}}}

\newcommand{\Auteq}{\operatorname{Auteq}}

\newcommand{\res}{\operatorname{res}}

\newcommand{\coh}{\operatorname{coh}}

\renewcommand{\mod}{\operatorname{mod}}

\newcommand{\CM}{\operatorname{CM}}

\newcommand\Db{\mathop{\rm{D}^b}}
\newcommand\perf{\operatorname{per}}

\newcommand{\Stab}{\operatorname{Stab}}
\newcommand{\cStab}[1]{\mathrm{Stab}_{#1}^{\kern -0.5pt \circ}\kern -0.25pt}
\newcommand{\TitsR}{{\sf Tits}_{\kern 1pt \bR}}
\newcommand{\CTitsR}{{\sf CTits}_{\kern 1pt \bR}}
\newcommand{\CTitsC}{{\sf CTits}_{\kern 1pt \bC}}
\newcommand{\cAut}{\mathrm{Aut}^{\kern 0pt \circ}\kern -0.25pt}

\def\Br{\mathop{\sf PBr}\nolimits}

\newcommand{\cA}{\mathcal{A}}

\newcommand{\cJ}{\mathcal{J}}

\newcommand{\bC}{\mathbb{C}}

\newcommand{\bH}{\mathbb{H}}

\newcommand{\bR}{\mathbb{R}}

\newcommand{\bZ}{\mathbb{Z}}

\newcommand{\Mut}{\operatorname{Mut}}

\newcommand{\Flop}{\operatorname{\sf Flop}}

\newcommand{\scrA}{\EuScript{A}}

\newcommand{\scrC}{\EuScript{C}}

\newcommand{\scrE}{\EuScript{E}}
\newcommand{\scrF}{\EuScript{F}}

\newcommand{\scrH}{\EuScript{H}}

\newcommand{\scrO}{\EuScript{O}}
\newcommand{\scrP}{\EuScript{P}}
\newcommand{\scrQ}{\EuScript{Q}}
\newcommand{\scrS}{\EuScript{S}}
\newcommand{\scrT}{\EuScript{T}}
\newcommand{\scrU}{\EuScript{U}}
\newcommand{\scrV}{\EuScript{V}}
\newcommand{\scrX}{\EuScript{X}}

\newcommand{\dsG}{\mathds{G}}

\newcommand{\simto}{\xrightarrow{\sim}}

\newcommand{\Delt}{\Updelta_{\kern 0.05em 0}}
\newcommand{\WDelt}{W_{\kern -0.1em \Updelta}}
\newcommand{\WDeltaff}{W_{\kern -0.1em \Updelta_{\aff}}}
\newcommand{\Wkern}[1]{W_{\kern -0.1em #1}\kern 0.05em}
\newcommand{\wo}[1]{w_{\kern -0.075em #1}}
\newcommand{\wop}[1]{w^{\phantom J}_{\kern -0.1em #1}}
\newcommand{\GammaJ}{\Upgamma_{\kern -0.05em J}}
\newcommand{\GammaS}{\Upgamma_{\kern -0.05em \cJ}}

\newcommand{\xlup}[1]{{}^{#1}\kern -0.15em x}
\newcommand{\xlupmax}[1]{{}^{#1}\kern -0.25em x}
\newcommand{\iDelta}{\iota_{\kern -0.075em \Updelta}}
\newcommand{\Upphisub}[1]{\Upphi_{\kern -0.1em #1}}
\newcommand{\aff}{\operatorname{\mathsf{aff}}\nolimits}

\newcommand{\con}{\mathrm{con}}
\renewcommand{\AA}{\mathrm{A}}
\newcommand{\BB}{\mathrm{B}}

\newcommand{\CA}{\AA_{\con}}
\newcommand{\CB}{\BB_{\con}}

\def\Id{\mathop{\rm{Id}}\nolimits}

\begin{document}

\title{Stability conditions for contraction algebras}

\author{Jenny August}
\address{Jenny August, Department of Mathematics, Aarhus University, Ny Munkegade 118, 8000 Aarhus C, Denmark.}
\email{jennyaugust@math.au.dk}
\author{Michael Wemyss}
\address{Michael Wemyss, The Mathematics and Statistics Building, University of Glasgow, University Place, Glasgow, G12 8QQ, UK.}
\email{michael.wemyss@glasgow.ac.uk}

\begin{abstract} 

This paper gives a description of the full space of Bridgeland stability conditions on the bounded derived category of a contraction algebra associated to a $3$-fold flop. The main result is that the stability manifold is the universal cover of a naturally associated hyperplane arrangement, which is known to be simplicial, and in special cases is an ADE root system. There are four main corollaries: (1) a  short proof of faithfulness of pure braid group actions in both algebraic and geometric settings, the first that avoid normal forms, (2) a classification of tilting complexes in the derived category of a contraction algebra, (3) contractibility of the stability space $\cStab{}\scrC$ associated to the flop, and (4) a new proof of the $K(\uppi,1)$-theorem in various finite settings, which includes ADE braid groups.
\end{abstract}

\thanks{J.A. was supported by the Carnegie Trust for the Universities of Scotland, and M.W. was supported by EPSRC grants~EP/R009325/1 and EP/R034826/1.}
\maketitle{}

\section{Introduction}

This paper provides a description of the entire stability manifold for a large class of finite dimensional algebras, namely contraction algebras associated to 3-fold flops, in terms of their $g$-vector fans. To establish this purely algebraic result, which takes place in the derived category of a finite dimensional algebra, requires both algebraic techniques through silting-discreteness, and also more remarkably geometric techniques through Fourier--Mukai transforms. The main corollaries are algebraic, topological and geometric.

\subsection{Change of categories} In birational geometry, $3$-fold flops play a vital role in the minimal model programme, providing a method to pass between the different minimal models of a given $3$-fold singularity. Contraction algebras \cite{DW1, DW3} are an uncountable class of symmetric, finite dimensional algebras introduced as a tool to study this geometry. To each 3-fold flopping contraction $f\colon X\to \Spec R$, where $X$ has at worst Gorenstein terminal singularities, there is an associated contraction algebra $\Lambda_\con$, which recovers all known numerical invariants of the flop \cite{DW1, HuaToda, HomMMP}.  A common theme, to which this paper contributes, is that results for contraction algebras imply results for $3$-fold flops.  


To explain this connection, note that associated to $f$ is the subcategory
\[
\scrC\colonequals \{ \scrF\in\Db(\coh X)\mid \mathbf{R
}f_*\scrF=0\}
\]
of $\Db(\coh X)$. This subcategory is small enough in the sense that it has desirable properties such as finite dimensional Hom-spaces, and the 3-Calabi--Yau property whenever $X$ is smooth, but it is large enough that many problems involving $X$, some of which are introduced in \S\ref{sec: intro cors} below, can be solved by replacing $\Db(\coh X)$ with $\scrC$.

The third, and easiest of all, category associated to this situation is the bounded derived category $\Db(\Lambda_\con)$.  This is very different from $\scrC$, but the transfer between them is provided by noncommutative deformation theory, which gives a universal sheaf $\scrE$ \cite{DW3} and a functor
\begin{equation}
-\otimes^{\bf L}_{\Lambda_\con}\!\scrE\colon \Db(\Lambda_\con)\to\Db(\coh X) \label{transfer}
\end{equation}
whose image lies in $\scrC$.  The main idea is that establishing results on $\Db(\Lambda_\con)$ is in fact easier, and this in turn establishes results on $\scrC$.  To do this, we use the functor \eqref{transfer} twice: first to bring Fourier--Mukai transforms into algebra to prove the main result on $\Db(\Lambda_\con)$, then second to transfer the consequences back to $\scrC$, and thus also $X$, in a series of corollaries.

\subsection{Main result}

The change in category from $\scrC$ to $\Db(\Lambda_\con)$ brings one major advantage, namely that $\Lambda_\con$ is silting-discrete \cite[4.12]{August1}. This fact has two happy consequences.  First, we are able to describe the \emph{full} space of stability conditions on $\Db(\Lambda_\con)$, not just a component. Second, it is known that all silting-discrete algebras have contractible stability manifolds \cite{PSZ}, and thus $\Stab\Db(\Lambda_\con)$ is contractible, before we even begin.

Alas, the change in category also brings one major disadvantage.  Whilst contractibility comes for free, moving to finite dimensional algebras means that we lose the technology of Fourier--Mukai transforms, and so controlling standard equivalences becomes significantly more difficult. 
The age-old algebraic problem of knowing whether a given autoequivalence which is the identity on simples is globally the identity functor rears its head.   Happily, in our setting we tame this problem by appealing to a commutative diagram that uses \eqref{transfer} to intertwine our algebraic equivalences with the geometric flop functors, where we can use a standard Fourier--Mukai argument (see Theorem \ref{functor is identity}).

The tilting theory of $\Lambda_{\con}$ is controlled by a real simplicial hyperplane arrangement $\scrH$ \cite{August2}, which can be described in many ways. On one hand, $\scrH$ arises as the $g$-vector fan of $\Lambda_{\con}$, and on the other it arises naturally as the ample cone of $f\colon X\to \Spec R$.  Our main result is the following.
\begin{thm}[\ref{contractible alg}]\label{stab main intro}
If $\Lambda_\con$ is the contraction algebra of  $f$, then the natural map
\[
\Stab{}\Db(\Lambda_\con)\to \scrX=\bC^{n}\backslash\scrH_{\bC}
\]
is a regular cover.  Furthermore, $\Stab{}\Db(\Lambda_\con)$ is contractible and so this cover is universal.
\end{thm}

The chambers of $\scrH$ can be labelled by the contraction algebras of flopping contractions reached from $f$ by iterated flops, and each path $\upalpha$ in the skeleton graph of $\scrH$ can be assigned a standard derived equivalence $F_\upalpha$. With this notation, the Galois group for the cover in Theorem~\ref{stab main intro} is the image of a \emph{pure} braid group under the group homomorphism $\uppi_1(\scrX)\to\Auteq\Db(\Lambda_\con)$ sending $\upalpha \mapsto F_\upalpha$.


\subsection{Corollaries}\label{sec: intro cors}
The universality in Theorem~\ref{stab main intro} allows us to both extract new results, and to simplify others in the literature.  The first consequence is the following, which gives a short proof of \cite[1.4]{August2}.  


\begin{cor}[\ref{faith cont alg}]\label{faith cont alg intro}
The homomorphism $\uppi_1(\scrX)\to\Auteq\Db(\Lambda_\con)$ sending $\upalpha\mapsto F_\upalpha$ is injective.
\end{cor}

Combining the topological Corollary~\ref{faith cont alg intro} with the algebraic-geometric Theorem~\ref{functor is identity}  allows us to side-step algebra automorphism issues, and to fully classify one-sided tilting complexes in $\Db(\Lambda_\con)$.  The following is a strengthening of \cite[7.2]{August2}.

\begin{cor}[\ref{cont 1-1}]
There is bijection between morphisms in the Deligne groupoid of $\scrH$ ending at the vertex associated to $\Lambda_{\con}$ and the set of isomorphism classes of basic one-sided tilting complexes of $\Lambda_{\con}$. 
\end{cor}

%
%
%

A more surprising consequence of the universality in Theorem~\ref{stab main intro} is the faithfulness of the geometric action $\uppi_1(\scrX)\to\Auteq\Db(\coh X)$ from \cite{DW3}. The logic is straightforward: given any element in the kernel of the geometric action, there is an element in the kernel of the algebraic action, and we can appeal directly to Corollary~\ref{faith cont alg intro}.   As such, the following recovers the main result of the paper \cite{HW}. 

\begin{cor}[\ref{faithful geo cor}]\label{faithful geo cor intro}
The homomorphism $\uppi_1(\scrX)\to\Auteq\Db(\coh X)$, which sends $\upalpha$ to the corresponding composition of Bridgeland-Chen flop functors, is injective.
\end{cor}

We can further extend this to establish results on the stability manifold of the associated subcategory $\scrC \subset \Db(\coh X)$. It was recently established in \cite{HW2} that there is a component of the space of stability conditions, $\cStab{}\scrC$, and a regular covering map
\[
\cStab{}\scrC\to \bC^{n}\backslash\scrH_{\bC},
\]
with Galois group equal to the image of $\uppi_1(\scrX)\to\Auteq\Db(\coh X)$. Given the faithfulness of this action from Corollary~\ref{faithful geo cor intro}, which is implied from the silting-discrete contractibility theorem, we are also able to deduce that $\cStab{}\scrC$ is contractible.  The logic again is straightforward:  universal covers are unique.  The proof of the following corollary is the only part of the paper where we use any prior results about $\cStab{}\scrC$.  

\begin{cor}[\ref{StabC contractible}]\label{StabC contractible intro}
$\cStab{}\scrC$ is contractible.
\end{cor}

The original proof of Corollary~\ref{StabC contractible intro} in \cite{HW2} relies on knowing the universal cover of $\bC^{n}\backslash\scrH_{\bC}$ is contractible, or in other words, knowing the $K(\uppi,1)$-conjecture holds for $\scrH$ (or more precisely its associated braid group). Since the arrangement $\scrH$ associated to $f$ can be obtained from the root system of an ADE Dynkin diagram by intersecting the reflecting hyperplanes with a certain subspace, it is always simplicial, and hence it is well-known the $K(\uppi,1)$-conjecture holds in this case \cite{Deligne}.  However, our approach via silting-discreteness allows us to bypass this result and actually reprove it without using \cite{B3,BT}, or normal forms. For example, if $\Lambda_{\con}$ is a contraction algebra with associated $\scrH$ being an ADE root system, then $\bC^{n}\backslash\scrH_{\bC}=\mathfrak{h}_{\mathrm{reg}}$ and it follows from Theorem~\ref{stab main intro} that the composition 
\[
\Stab\Db(\Lambda_\con)\to \mathfrak{h}_{\mathrm{reg}}\to \mathfrak{h}_{\mathrm{reg}}/W
\]
is also a covering map.

\begin{cor}[\ref{kpi1main}] \label{kpi1intro}
The $K(\uppi,1)$-conjecture holds for all ADE braid groups.
\end{cor}

We recall the $K(\uppi,1)$-conjecture in Section~\ref{Section 5} and remark that we actually reprove it for all intersection arrangements obtained from an ADE root system by intersecting the reflecting hyperplanes with a subspace generated by a choice of the coordinate vectors. This class of intersection arrangements is a bit eclectic: as a consequence, we prove $K(\uppi,1)$ for the Coxeter groups $I_n$ for \mbox{$n=3,4,5,6,8$}, but none of the other $n$.  The paper \cite{IW9} describes in more detail the types of arrangements that can arise; see Remark~\ref{intersection arrangement remark}.

\subsection*{Conventions}
Throughout we work over the field of complex numbers. Given a noetherian ring $A$, modules will be right modules unless specified, and $\mod A$ denotes the category of finitely generated $A$-modules.  We use the functional convention for composing arrows, so $f\circ g$ means $g$ then $f$.  With this convention, given a ring $R$, an $R$-module $M$ is a left $\End_R(M)$-module. Furthermore, $\Hom_R(M,N)$ is a right $\End_R(M)$-module and a left $\End_R(N)$-module, in fact a bimodule.


\section{Wall Crossing and Functorial Composition}

Throughout this paper $f\colon X\to \Spec R$ is a fixed 3-fold flopping contraction, where $X$ has at worst Gorenstein terminal singularities, and $R$ is complete local.  Necessarily $R$ is an isolated cDV singularity \cite{Pagoda}.  Associated to $f$ is a rigid Cohen--Macaulay (=CM) module $N$, an algebra $\Lambda\colonequals\End_R(N)$, and a contraction algebra $\Lambda_\con$.  We first briefly review these notions, mainly to set notation.

\subsection{Rigid Modules}\label{rigid prelim}
Since $R$ is Gorenstein, recall that 
\[
\CM R\colonequals\{ X\in\mod R\mid\Ext^i_R(X,R)=0\mbox{ for all }i>0\}.
\] 
For $X\in\CM R$, we say that $X$ is \emph{basic} if there are no repetitions in its Krull--Schmidt decomposition into indecomposables.  We call $X$ a \emph{generator} if one of these indecomposable summands is $R$, and we call $X$  \emph{rigid} if $\Ext^1_R(X,X)=0$.  We say $X\in\CM R$ is \emph{maximal rigid} if it is rigid, and furthermore it is maximal with respect to this property (see \cite[4.1]{IW4}).
 
 By the Auslander--McKay correspondence for cDV singularities, there is a one-to-one correspondence between flopping contractions $Y\to \Spec R$, up to $R$-isomorphism, and basic rigid generators in $\CM R$  \cite[4.13]{HomMMP}.   In particular, both sets are finite.
 For our fixed flopping contraction $f\colon X\to\Spec R$, the corresponding basic rigid CM generator across the bijection will throughout this paper be denoted $N$. The \emph{contraction algebra} of $f$ may then be defined as the stable endomorphism algebra $\Lambda_{\con} \colonequals \uEnd_R(N)$ \cite{DW1,DW3}.

The set of basic rigid CM generators carries an operation called \emph{mutation}.  Indeed, given such an $L$, with indecomposable summand $L_i\ncong R$, there is the so-called exchange sequence 
\begin{equation}
0\to L_i\to\bigoplus_{j\neq i}{L_j}^{\oplus b_{ij}} \to K_i\to 0\label{bij}
\end{equation}
and $\upnu_iL\colonequals \frac{L}{L_i}\oplus K_i$ \cite[(A.A)]{HomMMP}.  Given the summand $K_i$ of $\upnu_iL$ we can mutate again, and obtain $\upnu_i\upnu_iL$. It is a general fact for isolated cDV singularities that $\upnu_i\upnu_iL\cong L$ \cite[4.20(1)]{HomMMP}, and moreover in the second exchange sequence
\[
0\to K_i\to\bigoplus_{j\neq i}{L_j}^{\oplus c_{ij}} \to L_i\to 0
\]
we have $b_{ij}=c_{ij}$ for all $i,j$ \cite[5.22]{HomMMP}.  Furthermore, there are equivalences of derived categories
\begin{align*}
\RHom(\Hom_R(L,\upnu_iL),-)&\colon\Db(\mod\End_R(L))\xrightarrow{\sim}\Db(\mod\End_R(\upnu_iL))\\
\RHom(\Hom_R(\upnu_iL,L),-)&\colon\Db(\mod\End_R(\upnu_iL))\xrightarrow{\sim}\Db(\mod\End_R(L)).
\end{align*}
We will abuse notation and notate both as $\Upphi_i$, and refer to them as the \emph{mutation functors}.

Given our fixed basic rigid CM generator $N$ corresponding to $f$, write $\Mut_0(N)$ for the set of basic rigid CM generators that can be obtained from $N$ by iteratively mutating at indecomposable summands noting that, in this paper, we will never mutate at the summand $R$.  Geometrically, across the Auslander--McKay correspondence, this corresponds to all $R$-schemes that can be obtained from $X$ by iteratively flopping curves.

\subsection{Hyperplanes and Labels}\label{hyper subsection and order}
We fix a decomposition $N=R\oplus N_1\oplus \hdots\oplus N_n$, and will often implicitly declare $N_0=R$. For every $L\in\Mut_0(N)$, and each indecomposable summand $L_i$ of $L$, there is an exact sequence
\begin{equation}
0\to \bigoplus_{j=0}^n N_j^{\oplus a_{ij}}\to \bigoplus_{j=0}^n N_j^{\oplus b_{ij}}\to L_i\to 0 \label{two-term approximation}
\end{equation}
where the first and second terms do not share any indecomposable summands. In the setting where $N\in\CM R$ is maximal rigid, this fact is very-well known; in the setting here for rigid generators we rely instead on \cite[9.29, 5.6]{IW9}.  In any case, consider the cone
\[
C_L\colonequals \left\{ \sum_{i=1}^n \upvartheta_i \left( \sum_{j=1}^n (b_{ij}-a_{ij}) \,\mathbf{e}_i \right) \,\middle\vert\, \upvartheta_i>0 \mbox{ for all }i\right\}\subset\mathbb{R}^n.
\] 
It is clear that $C_N=\{ \sum_{i=1}^n \upvartheta_i \mathbf{e}_i\mid \upvartheta_i>0 \mbox{ for all }i\}$ which we denote throughout by $C_+$.  Furthermore, the chambers $C_L$, as $L$ varies over $\Mut_0(N)$, sweep out the chambers of a simplicial hyperplane arrangement $\scrH$ \cite{HW, IW9}.   


\begin{rem} \label{intersection arrangement remark}
The above arrangement has several equivalent descriptions; it can be obtained as the $g$-vector fan of $\Lambda_{\con}$ \cite[7.1]{August2} or as the moveable cone of the flopping contraction $f$ \cite[4.6(3)]{HW} (where, as described in the following remark, both of these results are extended to the rigid setting using \cite[9.29]{IW9}). Using this last description, HomMMP \cite[5.24, 5.25]{HomMMP} gives the description of $\scrH$ via intersecting a subspace of an ADE root system with the  reflecting hyperplanes. In other words, all the $\scrH$ above are intersection arrangements of ADE root systems, and as such, are both finite and simplicial.  
\end{rem}

\begin{rem}
The three key properties appearing here and later are the following.
\begin{enumerate}
    \item $\upnu_i\upnu_iL \cong L$ for any basic rigid CM generator.
    \item Moreover, $b_{ij}=c_{ij}$ in the corresponding exchange sequences.
    \item Every $L\in\Mut_0(N)$ has a two-term approximation by $N$, as in \eqref{two-term approximation}, ensuring that $\Hom_R(L,N)$ is a tilting bimodule for $\Lambda$.
\end{enumerate} 
The first two follow as a consequence of HomMMP \cite[4.20(1), 5.22]{HomMMP}, since flopping is an involution.  In the maximal rigid setting (3) is known, but to prove (3) in the general rigid setting requires a small part of \cite{IW9}.    Once (1)--(3) hold, \cite[4.9(2)]{August2} extends all the results of \cite{HW, August2} to the setting of rigid modules, and this is what we use below.  A reader wishing only to work in the maximal rigid setting, or in the smooth setting, may safely disregard references to \cite{IW9}.
\end{rem}

 As explained by the general Coxeter-style labelling of walls and chambers in \cite[\S9.7]{IW9}, the fixed decomposition $N=R\oplus N_1\oplus \hdots\oplus N_n$ induces an ordering on the summands of all other elements $L$ of $\Mut_0(N)$, such that each local wall crossing has a label $s_i$ and corresponds to mutation at the $i$th summand. In this way, there is a compatible labelling on the edges of the 1-skeleton of $\scrH$.  We illustrate this in one example.
\begin{exa}
There exists a two-curve flop over a $cD_4$ singularity with the simplicial hyperplane arrangement shown in Figure~\ref{figure1}.  Each local wall crossing is labelled by $s_i$ for $i=1,2$, and under crossing wall $s_i$ the $i$th summand is mutated.  For clarity, in each chamber we have not written the zeroth summand $R$.    
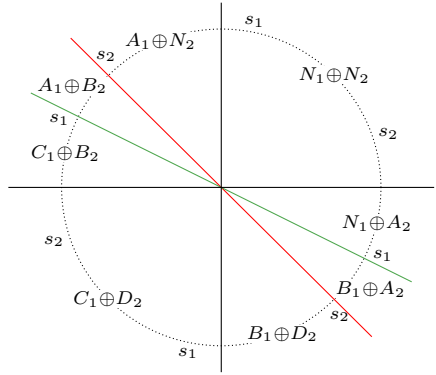
\begin{figure}[h]
\[
\begin{array}{ccc}
\begin{array}{c}
\begin{tikzpicture}[scale=1.4,bend angle=20]
\draw[densely dotted] (0,0) circle (1.5cm);
\coordinate (A1) at (135:2cm);
\coordinate (A2) at (-45:2cm);
\coordinate (B1) at (153.435:2cm);
\coordinate (B2) at (-26.565:2cm);
\coordinate (C1) at (161.565:2cm);
\coordinate (C2) at (-18.435:2cm);
\draw[red!60] (A1) -- (A2);
\draw[green!50!black!50] (B1) -- (B2);
\draw[black!60] (-2,0)--(2,0);
\draw[black!60] (0,-1.75)--(0,1.75);
\node (C+) at (45:1.5cm)[gap] {$\scriptstyle N_1\oplus N_2$};
\node (C1) at (112.5:1.5cm)[gap] {$\scriptstyle A_1\oplus N_2$};
\node (C2) at (145.5:1.7cm)[gap] {$\scriptstyle A_1\oplus B_2$};
\node (C3) at (167.5:1.5cm)[gap] {$\scriptstyle C_1\oplus B_2$};
\node (C-) at (225:1.5cm)[gap] {$\scriptstyle C_1\oplus D_2$};
\node (C4) at (-67.5:1.5cm)[gap] {$\scriptstyle B_1\oplus D_2$};
\node (C5) at (-34.5:1.7cm)[gap] {$\scriptstyle B_1\oplus A_2$};
\node (C6) at (-13:1.5cm)[gap] {$\scriptstyle N_1\oplus A_2$};
\node at (78.75:1.6cm) {$\scriptstyle \upnu_1$};
\node at (131:1.65cm) {$\scriptstyle \upnu_2$};
\node at (157:1.65cm) {$\scriptstyle \upnu_1$};
\node at (198:1.65cm) {$\scriptstyle \upnu_2$};
\node at (258.75:1.6cm) {$\scriptstyle \upnu_1$};
\node at (312.5:1.65cm) {$\scriptstyle \upnu_2$};
\node at (337:1.65cm) {$\scriptstyle \upnu_1$};
\node at (378:1.65cm) {$\scriptstyle \upnu_2$};
\end{tikzpicture}
\end{array}
&\quad&
\begin{array}{c}
\begin{tikzpicture}[scale=1.4,bend angle=20]
\draw[densely dotted] (0,0) circle (1.5cm);
\coordinate (A1) at (135:2cm);
\coordinate (A2) at (-45:2cm);
\coordinate (B1) at (153.435:2cm);
\coordinate (B2) at (-26.565:2cm);
\coordinate (C1) at (161.565:2cm);
\coordinate (C2) at (-18.435:2cm);
\draw[red!34] (A1) -- (A2);
\draw[green!50!black!40] (B1) -- (B2);
\draw[black!40] (-2,0)--(2,0);
\draw[black!40] (0,-1.75)--(0,1.75);
\node (C+) at (45:1.5cm)[B] {};
\node (C1) at (112.5:1.5cm)[B] {};
\node (C2) at (145.5:1.5cm)[B] {};
\node (C3) at (167.5:1.5cm)[B] {};
\node (C-) at (225:1.5cm)[B] {};
\node (C4) at (-67.5:1.5cm)[B] {};
\node (C5) at (-34.5:1.5cm)[B] {};
\node (C6) at (-13:1.5cm)[B] {};
\node at (78.75:1.6cm) {$\scriptstyle s_1$};
\node at (131:1.65cm) {$\scriptstyle s_2$};
\node at (157:1.65cm) {$\scriptstyle s_1$};
\node at (198:1.65cm) {$\scriptstyle s_2$};
\node at (258.75:1.6cm) {$\scriptstyle s_1$};
\node at (312.5:1.65cm) {$\scriptstyle s_2$};
\node at (337:1.65cm) {$\scriptstyle s_1$};
\node at (378:1.65cm) {$\scriptstyle s_2$};
\end{tikzpicture}
\end{array}\\
\mbox{Exchange graph}
&&
\mbox{1-skeleton}
\end{array}
\]
\caption{Exchange graph and $1$-skeleton for certain two-curve $cD_4$ singularity.}
\label{figure1}
\end{figure}
Note that $C_N=C_+$ is the top right chamber.
\end{exa}

Fixing an ordering of the summands of $N$ not only fixes an ordering of the projective modules $\scrP_i=\Hom_R(N,N_i)$ of $\Lambda=\End_R(N)$, and via the pairing between simples and projectives, an ordering of its simples $\scrS_0,\scrS_1,\hdots,\scrS_n$, but it also fixes an ordering on the projectives and simples of $\End_R(L)$ for all $L\in\Mut_0(N)$.

To fix notation, for $L\in\Mut_0(N)$ suppose that the induced ordering on the summands of $L$ is $L=L_0\oplus L_1\oplus\hdots\oplus L_n$, where $L_0= R$.  There is an induced ordering on the projectives $\scrQ_i\colonequals \Hom_R(L,L_i)$ of $\End_R(L)$, and again via the pairing between simples and projectives, an ordering of its simples $\scrS'_0,\scrS'_1,\hdots,\scrS'_n$.   The simples for the contraction algebra $\uEnd_R(L)$ are $\scrS'_1,\hdots,\scrS'_n$.

\subsection{Standard Equivalences for Simple Wall Crossings}
Let $L\in\Mut_0(N)$, with associated $\AA \colonequals \End_R(L)$ and contraction algebra  $\CA \colonequals \uEnd_R(L)$.  To ease notation, suppose that $\BB \colonequals \End_R(\upnu_iL)$, with contraction algebra $\CB$, and consider the good truncation 
\begin{align}
\scrT_i \colonequals \uptau_{\scriptscriptstyle{ \geq -1}} ( \CB \otimes^{\bf L}_\BB \Hom_R(L,\upnu_iL) \otimes^{\bf L}_{\AA} \CA), \label{bimodcomplex}
\end{align}
which is a complex of $\CB$-$\CA$ bimodules. In this basic rigid setting, it is already known that the complex $\scrT_i$ is a two-sided tilting complex \cite[3.3]{August2}.

\begin{thm}\label{intertwinement}
Suppose that $L,M \in \Mut_0(N)$, which correspond to the flopping contractions $X_L\to\Spec R$ and $X_M\to\Spec R$ say.  Then the following are equivalent.
\begin{enumerate}
\item $C_L$ and $C_M$ share a codimension one wall in $\scrH$.
\item $M\cong\upnu_iL$ for some $i\neq 0$.
\item $X_L$ and $X_M$ are related by a flop at a single irreducible curve. 
\end{enumerate}
In this case, set $\AA \colonequals \End_R(L)$, $\CA \colonequals \underline{\End}_R(L)$, and $\BB \colonequals \End_R(\upnu_iL)$, $\CB \colonequals \underline{\End}_R(\upnu_iL)$. Then the mutation operations intertwine via the following commutative diagram
\[
\begin{tikzpicture}
\node (Acon) at (0,0) {$\Db(\CA)$};
\node (A) at (4,0) {$\Db(\AA)$};
\node (X) at (8,0) {$\Db(\coh X_L)$};
\node (Bcon) at (0,-2) {$\Db(\CB)$};
\node (B) at (4,-2) {$\Db(\BB)$};
\node (Y) at (8,-2) {$\Db(\coh X_{\upnu_iL})$};
\draw[->] (Acon) --
node[above] { $\scriptstyle \res$} (A);
\draw[->] (A) --
node[above] { $\scriptstyle - \otimes^{\bf L}_{\AA}\scrV$} node [below] {$\scriptstyle\sim$} (X);
\draw[->] (Bcon) --
 node[above] { $\scriptstyle \res$} (B);
\draw[->] (B) --
 node[above] { $\scriptstyle - \otimes^{\bf L}_{\BB} \scrV^+$} node [below] {$\scriptstyle\sim$} (Y);
 
\draw[->] (Acon) --node [left] {$\scriptstyle F_{i} \colonequals \RHom_{\CA}(\scrT_i, -)$} node [right] {$\scriptstyle\sim$} (Bcon);
\draw[->] (A) --node [left] {$\scriptstyle \Upphi_{i}$} node [right] {$\scriptstyle\sim$} (B);
\draw[->] (X) -- node[left] {$ \scriptstyle  \Flop_i$} node [right] {$\scriptstyle\sim$} (Y);
\end{tikzpicture}
\]
where $\Upphi_{i}$ is the mutation functor, $\Flop_i$ is the inverse of the Bridgeland--Chen flop functor \cite{Bridgeland, Chen}, $\scrV$ and $\scrV^+$ are the standard projective generators of zero perverse sheaves \cite{VdB}, and $\res$ is restriction of scalars induced from the ring homomorphisms $\AA\to\CA$ and $\BB\to\CB$.
\end{thm}
\begin{proof}
The statement (1)$\Leftrightarrow$(2) is clear; see also \cite[3.6]{HW2}. The statement (2)$\Leftrightarrow$(3) is \cite[4.13]{HomMMP}, where since we are mutating only at indecomposable summands and $R$ is isolated, \cite[4.20(1)]{HomMMP} overrides the caveat in the latter part of \cite[4.13]{HomMMP}.    The left hand diagram commutes by \cite[1.1]{August2}, and the right hand diagram commutes by \cite[4.2]{HomMMP}.
\end{proof}

\subsection{K-theory}
Since $\Lambda_\con$ is a finite dimensional algebra, consider the Grothendieck group
\[
\mathsf{G}_0(\Lambda_\con)\colonequals K_0(\Db(\Lambda_\con)),
\]
which is well-known to be a free abelian group based by the ordered simples $[\scrS_1],\hdots,[\scrS_n]$.  The notation $\mathsf{G}_0$ is chosen since we are including in the Grothendieck group all modules, whereas $K_0$ often only deals with vector bundles, and thus projectives.

Continuing the notation from Theorem~\ref{intertwinement}, for any $L\in\Mut_0(N)$ consider its associated contraction algebra $\CA=\uEnd_R(L)$ and its mutations $\upnu_i\CA\colonequals\uEnd_R(\upnu_iL)$. Abusing notation slightly, consider the standard equivalences
\[
\begin{tikzpicture}
\node at (1,0) {$\Db(\CA)$};
\node at (4.15,0) {$\Db(\upnu_i\CA)$};
\draw[->] (2,0.1) -- node[above] {$\scriptstyle F_i$}(3,0.1); 
\draw[<-] (2,-0.1) -- node[below] {$\scriptstyle F_i$}(3,-0.1);
\end{tikzpicture}
\]
where the functor from right to left is induced by the mutation $\upnu_i L \to \upnu_i \upnu_i L \cong L$.
These, and their inverses, induce the following four isomorphisms on K-theory
\begin{equation}
\begin{array}{c}
\begin{tikzpicture}
\node at (1.1,0) {$\mathsf{G}_0(\CA)$};
\node at (4.1,0) {$\mathsf{G}_0(\upnu_i\CA)$};
\draw[->] (2,0.1) -- node[above] {$\scriptstyle \mathsf{F}_i$}(3,0.1); 
\draw[<-] (2,-0.1) -- node[below] {$\scriptstyle \mathsf{F}_i$}(3,-0.1);
\end{tikzpicture}
\end{array}
\qquad
\begin{array}{c}
\begin{tikzpicture}
\node at (1.1,0) {$\mathsf{G}_0(\CA)$};
\node at (4.1,0) {$\mathsf{G}_0(\upnu_i\CA).$};
\draw[<-] (2,0.1) -- node[above] {$\scriptstyle \mathsf{F}_i^{-1}$}(3,0.1); 
\draw[->] (2,-0.1) -- node[below] {$\scriptstyle \mathsf{F}_i^{-1}$}(3,-0.1);
\end{tikzpicture}
\end{array}
\label{four matrices}
\end{equation}

As in Subsection~\ref{hyper subsection and order}, write $\{\scrS'_1,\hdots,\scrS'_n\}$ for the ordered simples of $\CA$.  For lack of suitable alternatives, also write $\{\scrS'_1,\hdots,\scrS'_n\}$ for the correspondingly ordered simples in $\upnu_i\CA$. 

\begin{lem}\label{k-correspond simples}
With the notation in \eqref{bij}, $\mathsf{F}_i\colon\mathsf{G}_0(\CA)\to\mathsf{G}_0(\upnu_i\CA)$ sends
\begin{equation}
[\scrS'_t]\mapsto
\left\{
\begin{array}{cl}
-[\scrS'_i]& \mbox{if } t= i,\\
b_{it}[\scrS'_i]+ [\scrS'_t] & \mbox{if } t\neq i.
\end{array}
\right.
\label{k-correspond simp}
\end{equation}
\end{lem}
\begin{proof}
The fact that $\Upphi_i(\scrS'_i)=\scrS'_i[-1]$ is \cite[4.15]{HomMMP}.  Theorem~\ref{intertwinement} then implies that $F_i(\scrS'_i)$ maps, under restriction of scalars, to $\scrS'_i[-1]$.  It follows that $F_i(\scrS'_i)\cong\scrS'_i[-1]$ in $\Db(\CA)$, see e.g.\ \cite[6.6]{August2}. This establishes the top row.  

For the second row, applying $\Hom_R(L,-)$ to \eqref{bij} and using the rigidity of $L$ gives an exact sequence
\[
0\to \scrQ_i
\to \bigoplus_{j\neq i} \scrQ_j^{\oplus b_{ij}}
\to \Hom_R(L,K_i)
\to 0.
\]
Applying $\Hom_{\AA}(-,\scrS'_t)$ to this, with $t\neq i$, yields
\[
\RHom_{\AA}(\Hom_R(L,K_i),\scrS'_t)=\mathbb{C}^{\oplus b_{it}}.
\] 
Further, it is clear that $\RHom_{\AA}(\scrQ_j,\scrS'_t) \cong \Hom_{\AA}(\scrQ_j,\scrS'_t)$ is zero if $j\neq t$, and equals $\mathbb{C}$ if $j=t$.  Combining, we see that $\Upphi_i(\scrS'_t)=\RHom_{\AA}(\Hom_R(L,\upnu_iL),\scrS'_t)$ is a module, filtered by $b_{it}$ copies of $\scrS'_i$, and one copy of $\scrS'_t$.  Again, the left hand side of the commutative diagram in Theorem~\ref{intertwinement} then shows that $F_i(\scrS'_t)$ must also be a module, filtered by $b_{it}$ copies of $\scrS'_i$, and one copy of $\scrS'_t$.  The second row follows.
\end{proof}

\begin{rem}\label{dual proj basis}
Applying the above to $L=N$, basing $\mathsf{G}_0(\Lambda_\con)$ by the ordered simples $[\scrS_1],\hdots,[\scrS_n]$, the above transformation \eqref{k-correspond simp} assembles into a $n\times n$ matrix, with coefficients in $\mathbb{Z}$, representing the map $\mathsf{F}_i\colon\mathbb{Z}^n\to\mathbb{Z}^n$.  As is standard in linear algebra, the dual map
\[
\mathsf{F}_i^*\colon \mathsf{G}_0(\upnu_i\Lambda_\con)^*\to\mathsf{G}_0(\Lambda_\con)^*,
\]
where $\mathsf{G}_0(\Lambda_\con)^*$ has dual basis $\mathbf{e}_1,\hdots,\mathbf{e}_n$ say, is given by the transpose matrix.  But the transpose is precisely the transformation
\[
\mathbf{e}_t\mapsto
\left\{
\begin{array}{cl}
\mathbf{e}_t& \mbox{if } t\neq i,\\
-\mathbf{e}_i+\sum_{j\neq i}b_{ij}\mathbf{e}_j & \mbox{if } t=i.
\end{array}
\right.
\]
which is precisely the transformation $\upvarphi_i$ seen in moduli tracking \cite[\S5]{HomMMP}, or in K-theory of projectives in \cite[3.2]{HW2}.  In particular, it will be convenient to think of the dual basis $\mathbf{e}_1,\hdots,\mathbf{e}_n$ as being the basis $[\scrP_1],\hdots,[\scrP_n]$ of $\mathsf{K}_0(\perf\Lambda)/[\scrP_0]$, where $\perf\Lambda$ is the full subcategory of $\Db(\Lambda)$ consisting of perfect complexes. Then the transformation $\mathsf{F}_i^*$ can be identified with these transformations elsewhere in the literature, and we can then use those results freely.  Note that the hyperplane arrangement $\scrH$ from Subsection~\ref{hyper subsection and order} is defined in terms of $\mathbf{e}_i$, and so naturally lives in $\mathsf{G}_0(\Lambda_\con)^*$.
\end{rem}

By the above remark, the proof of the following does  follow as the dual of \cite[3.2]{HW2}.  It is however instructive to give a direct proof.

\begin{lem}\label{invol}
All four isomorphisms in \eqref{four matrices} are given by the same matrix, namely the one from \eqref{k-correspond simp}, and this matrix squares to the identity.  
\end{lem}
\begin{proof}
By \eqref{k-correspond simp}, the matrices are controlled by the numbers $b_{ij}$ appearing in the relevant exchange sequences.  Say the top $\mathsf{F}_i$ is controlled by numbers $b_{ij}$, and the bottom $\mathsf{F}_i$ is controlled by numbers $c_{ij}$.  That the two matrices labelled $\mathsf{F}_i$ are the same is simply the statement that $b_{ij}=c_{ij}$, which has already been explained in Subsection~\ref{rigid prelim}.   Given this fact that $b_{ij}=c_{ij}$, we see that $\mathsf{F}_i\mathsf{F}_i={\rm Id}$ by simply observing 
\[
[\scrS'_t ] \ \xmapsto{\mathsf{F}_i} \
\left\{
\begin{array}{cl}
-[\scrS'_i]& \mbox{if } t= i,\\
b_{it}[\scrS'_i]+ [\scrS'_t] & \mbox{if } t\neq i.
\end{array}
\right. \
\xmapsto{\mathsf{F}_i} \
\left\{
\begin{array}{cl}
-(-[\scrS'_i])& \mbox{if } t= i\\
-b_{it}[\scrS'_i]+ (b_{it}[\scrS'_i] + [\scrS'_t]) & \mbox{if } t\neq i,
\end{array}
\right.
\]
which is clearly the identity.  Applying $\mathsf{F}_i^{-1}$ to each side of the equation $\mathsf{F}_i\mathsf{F}_i={\rm Id}$ gives $\mathsf{F}_i^{-1}=\mathsf{F}_i$, and all statements follow.
\end{proof}

\subsection{Groupoids}\label{groupoids subsection}
As in Subsection~\ref{hyper subsection and order}, associated to every contraction algebra is a hyperplane arrangement $\scrH$.  As is standard, there is an associated graph $\Gamma_{\scrH}$ defined as follows.

\begin{dfn} 
The vertices of $\Gamma_{\scrH}$ are the chambers, i.e.\ the connected components, of $\bR^n\backslash\scrH$. There is a unique arrow $a \colon  v_1\to v_2$ from chamber $v_1$ to chamber $v_2$ if the chambers are adjacent, otherwise there is no arrow.
\end{dfn}
By definition, if there is an arrow $a \colon  v_1\to v_2$, then there is a unique arrow $b\colon  v_2\to v_1$ with the opposite direction of $a$.  For an arrow $a\colon  v_1\to v_2$, set $s(a)\colonequals  v_1$ and $t(a)\colonequals  v_2$.  
  
A \emph{positive path of length~$n$} in $\Gamma_{\scrH}$ is  a formal symbol
\[
p=a_n\circ \hdots\circ a_2\circ a_1,
\]
 whenever there exists a sequence of vertices $v_0,\hdots,v_n$ of $\Gamma_{\scrH}$ and arrows $a_i\colon v_{i-1}\to v_i$ in $\Gamma_{\scrH}$. Set $s(p)\colonequals  v_0$, $t(p)\colonequals  v_n$, $\ell(p)\colonequals  n$, and write $p\colon s(p)\to t(p)$.  If 
$q=b_m\circ\hdots\circ b_2 \circ b_1$ is another positive path with $t(p)=s(q)$, we consider the formal symbol
\[
q\circ p\colonequals  b_m\circ\hdots\circ b_2 \circ b_1
\circ
a_n\circ \hdots\circ a_2\circ a_1,
\]
and call it the {\it composition} of $p$ and $q$.

\begin{dfn}
A positive path $\upalpha$ is called \emph{minimal} if there is no shorter positive path in $\Gamma_{\scrH}$, with the same start and end points as $\upalpha$.
\end{dfn}
Following \cite[p7]{Delucchi}, let $\sim$ denote the smallest equivalence relation, compatible with morphism composition, that identifies all morphisms that arise as positive minimal paths with same source and target. Then consider the free category $\mathrm{Free}(\Gamma_\scrH)$ on the graph $\Gamma_\scrH$, whose morphisms correspond to directed paths, and its quotient category 
\[
\dsG^{+} \colonequals \mathrm{Free}(\Gamma_\scrH)/\sim, 
\]
called the category of positive paths.

\begin{dfn} 
The \emph{Deligne groupoid}  $\dsG$ is the groupoid defined as the groupoid completion of  $\dsG^{+}$, that is, a formal inverse is added for every morphism in $\dsG^{+}$.
\end{dfn} 

It is a very well known fact \cite{Deligne, Paris, Paris3, Salvetti} (see also \cite[2.1]{Paris2}) that for any vertex $v\in \dsG$ there is an isomorphism $\End_\dsG(v)\cong\uppi_1(\scrX)$ where $\scrX=\mathbb{C}^n \setminus \scrH_\mathbb{C}$. This fact is general, and is significantly weaker and easier to establish than statements involving $K(\uppi,1)$.  We will use this fact implicitly throughout. 

\subsection{Composition and K-theory}\label{def of Falpha}
For any $\upalpha\in\mathrm{Free}(\Gamma_\scrH)$, say $\upalpha={s_{i_t}}\circ\hdots\circ {s_{i_1}}$, consider
\begin{align*}
F_\upalpha&\colonequals F_{i_t}\circ\hdots\circ F_{i_1}\\
\Upphi_\upalpha&\colonequals\Upphi_{i_t}\circ\hdots\circ \Upphi_{i_1}.
\end{align*}
The following is known, and is easy to establish just using the tilting order, in the case when the modules are maximal rigid (e.g.\ if $X$ is smooth).  In our more general situation of rigid generators, the same proof works, but it relies on some recent advances in \cite{IW9}.    
\begin{prop}\label{terminal ok}
Let $\upalpha\colon C_L\to C_M$ be a positive minimal path.  Set $\AA \colonequals \End_R(L)$, $\BB \colonequals \End_R(M)$, $\CA \colonequals \underline{\End}_R(L)$ and $\CB\colonequals \underline{\End}_R(M)$. Then the following hold.
\begin{enumerate}
\item $\Upphi_\upalpha$ is functorially isomorphic to  $\RHom_{\AA}(\Hom_R(L,M),-)$.
\item $F_\upalpha$ is functorially isomorphic to $\RHom_{\CA}(\scrT_{LM},-)$ where $
\scrT_{LM}$ is the two-sided tilting complex $ \uptau_{\scriptscriptstyle{ \geq -1}} ( \CB \otimes^{\bf L}_\BB \Hom_R(L,M) \otimes^{\bf L}_{\AA} \CA)$.
\end{enumerate}
In particular, all positive minimal paths with the same start and end points are functorially isomorphic.
\end{prop}
\begin{proof}
(1) When $N$ is maximal rigid, this is precisely \cite[4.9(1)]{August2}, \cite[4.6]{HW}.  In the more general setting here with $L,M\in\Mut_0(N)$, then certainly $L\in\Mut_0(M)$. Since $R$ is isolated cDV, it follows from the combinatorial and geometric description of mutation of rigid modules in \cite[9.25, 9.29]{IW9} that $\Hom_R(L,M)$ is a tilting $\End_R(M)$-$\End_R(L)$-bimodule, of projective dimension one when viewed as a right $\End_R(L)$-module.  This is the key technical condition, explained in \cite[4.9(2)]{August2}, and not available when \cite{HW} was written, that now allows us to use the main result \cite[4.6]{HW} freely in the more general setting here. \\
(2) When $N$ is maximal rigid, this is \cite[4.12]{August2}. The more general statement required here  follows by (1), and again the Remark \cite[4.9(2)]{August2} which asserts, given (1), we are able to apply the main result \cite[4.12]{August2} to the setting of rigid objects.
\end{proof}

For any two positive minimal paths $\upalpha$ and $\upbeta$ with the same start and end points, the above proposition shows that there is a functorial isomorphism  $F_\upalpha\cong F_\upbeta$.  Hence the association $\upalpha\mapsto F_\upalpha$ descends to a functor from $\dsG^+$.  Since $F_\upalpha$ is already an equivalence, this in turn formally descends to a functor from $\dsG$.    Using the same logic,  the assignment $\upalpha\mapsto\Upphi_\upalpha$ also descends to a functor from $\dsG^+$, and then a functor from $\dsG$.

Furthermore, for every $\upalpha\in\Hom_{\dsG}(C_L,C_M)$, using the notation from the second sentence in Proposition~\ref{terminal ok}, the following diagram commutes
\begin{equation}
\begin{array}{c}
\begin{tikzpicture}
\node (Acon) at (0,0) {$\Db(\CA)$};
\node (A) at (4,0) {$\Db(\AA)$};
\node (Bcon) at (0,-2) {$\Db(\CB)$};
\node (B) at (4,-2) {$\Db(\BB)$};
\draw[->] (Acon) --
node[above] { $\scriptstyle \res$} (A);
\draw[->] (Bcon) --
 node[above] { $\scriptstyle \res$} (B);
\draw[->] (Acon) --node [left] {$\scriptstyle F_{\upalpha}$} node [right] {$\scriptstyle\sim$} (Bcon);
\draw[->] (A) --node [left] {$\scriptstyle \Upphi_{\upalpha}$} node [right] {$\scriptstyle\sim$} (B);
\end{tikzpicture}
\end{array}
\label{commutes for alpha}
\end{equation}
just by composition:  by \cite[1.1]{August2}, respectively \cite[3.2]{August2}, both of the following commute.
\[
\begin{array}{ccc}
\begin{array}{c}
\begin{tikzpicture}
\node (Acon) at (0,0) {$\Db(\CA)$};
\node (A) at (3.5,0) {$\Db(\AA)$};
\node (Bcon) at (0,-2) {$\Db(\upnu_i\CA)$};
\node (B) at (3.5,-2) {$\Db(\upnu_i\AA)$};
\draw[->] (Acon) --
node[above] { $\scriptstyle \res$} (A);
\draw[->] (Bcon) --
 node[above] { $\scriptstyle \res$} (B);
\draw[->] (Acon) --node [left] {$\scriptstyle F_{i}$} node [right] {$\scriptstyle\sim$} (Bcon);
\draw[->] (A) --node [left] {$\scriptstyle \Upphi_{i}$} node [right] {$\scriptstyle\sim$} (B);
\end{tikzpicture}
\end{array}
&&
\begin{array}{c}
\begin{tikzpicture}
\node (Acon) at (0,0) {$\Db(\CA)$};
\node (A) at (3.5,0) {$\Db(\AA)$};
\node (Bcon) at (0,-2) {$\Db(\upnu_i\CA)$};
\node (B) at (3.5,-2) {$\Db(\upnu_i\AA)$};
\draw[->] (Acon) --
node[above] { $\scriptstyle \res$} (A);
\draw[->] (Bcon) --
 node[above] { $\scriptstyle \res$} (B);
\draw[->] (Acon) --node [left] {$\scriptstyle F_{i}^{-1}$} node [right] {$\scriptstyle\sim$} (Bcon);
\draw[->] (A) --node [left] {$\scriptstyle \Upphi_{i}^{-1}$} node [right] {$\scriptstyle\sim$} (B);
\end{tikzpicture}
\end{array}
\end{array}
\]

Later, the following is one of the crucial ingredients in establishing that $\Stab\Db(\Lambda_\con)$ is a covering space. Recall that for the standard derived equivalence $F_\upalpha$ associated to a path $\upalpha$, the induced map on the K-theory is denoted $\mathsf{F}_\upalpha$.
 
\begin{prop}\label{trivial K} 
Suppose that $\upbeta\colon C\to D$ is a positive minimal path in $\mathrm{Free}(\Gamma_\scrH)$.
\begin{enumerate}
\item\label{trivial K 1} 
If $\upalpha\colon C\to D$ is any positive path, then $\mathsf{F}_{\upalpha}=\mathsf{F}_{\upbeta}$. 
 \item\label{trivial K 2} If $\upalpha\in\End_{\dsG}(C)$, then $\mathsf{F}_{\upalpha}={\rm Id}$.
 \item\label{trivial K 3} If $\upalpha,\upgamma\in\Hom_{\dsG}(C,D)$, then $\mathsf{F}_{\upalpha}=\mathsf{F}_{\upgamma}$.
 \end{enumerate}
\end{prop}
\begin{proof}
Having established Lemma~\ref{invol}, this is now word-for-word identical to \cite[4.8]{HW2}.   Note that this proof is elementary, and does not require Deligne normal form.  
\end{proof}

\subsection{The Dual Composition}\label{dual comp subsection}
For each $L \in \Mut_0(N)$,  consider $\Lambda_L\colonequals\End_R(L)$ and recall that $\Lambda \colonequals \End_R(N)$.  Choose a positive minimal path $\upbeta \colon C_L \to C_+$, which in turn gives rise to a derived equivalence
\[
\Upphi_L \colonequals \Upphi_\upbeta 
\stackrel{\scriptstyle\ref{terminal ok}}{\cong} \RHom_{\Lambda_L}(\Hom_R(L,N), -) \colon \Db(\mod\Lambda_L) \to \Db(\mod\Lambda).
\]
This derived equivalence is independent of choice of positive minimal path, by Proposition~\ref{terminal ok}.  It induces an isomorphism $\mathsf{K}_0(\perf\Lambda_L)\to\mathsf{K}_0(\perf\Lambda)$ on the K-theory of perfect complexes, so write
\[
[\Upphi_L(\scrQ_i)]=\sum_{j=0}^n(\upvarphi_L)_{ij}[\scrP_j]
\]
in $\mathsf{K}_0(\perf\Lambda)\cong\mathbb{Z}^{n+1}$, where $\scrQ_i =\Hom_R(L,L_i)$ and $\scrP_i =\Hom_R(N,N_i)$.  Since $\upalpha$ is a sequence of mutations that do not involve mutating the zeroth summand, at each stage the zeroth summand is fixed.  Hence this isomorphism descends to an isomorphism 
\[
\upvarphi_L\colon \mathsf{K}_0(\perf\Lambda_L)/[\scrQ_0]\xrightarrow{\sim}\mathsf{K}_0(\perf\Lambda)/[\scrP_0].
\]
Basing the first by $[\scrQ_1],\hdots,[\scrQ_n]$ and the second by $[\scrP_1],\hdots,[\scrP_n]$, the matrix representing the isomorphism is $(\upvarphi_L)_{ij}$ for $1\leq i,j\leq n$.  By Remark~\ref{dual proj basis}, later we will think of these bases as $\mathbf{e}'_1,\hdots,\mathbf{e}'_n$ of $\mathsf{G}_0(\uEnd_R(L))^*$ and $\mathbf{e}_1,\hdots,\mathbf{e}_n$ of $\mathsf{G}_0(\Lambda_\con)^*$ respectively.

\begin{rem}
The above description of $\upvarphi_L$ is in terms of projectives of the ambient $\End_R(L)$, since this is convenient later.  There is however a much more intrinsic description of $\upvarphi_L$ that does not rely on this larger algebra, via the two-term tilting complexes of the contraction algebra $\Lambda_{\con}$. In particular, in the language of $g$-vectors, $\upvarphi_L(\mathbf{e}'_i) = g^{L_i}$, where $g^{L_i}$ is the $g$-vector of the two-term complex of $\Lambda_{\con}$ associated to the rigid object $L_i$ via the bijection \cite[2.18]{August1}. We do not use this description below.
\end{rem}

\section{Stability and t-structures}\label{Section 3}

\subsection{Stability Generalities}\label{stab gen} 
Throughout this subsection, $\scrT$ denotes a triangulated category whose Grothendieck group $K_0(\scrT)$ is a finitely generated free $\bZ$-module.

\begin{prop}[{\cite[5.3]{B07}}]\label{stab prop}
To give a stability condition on $\scrT$ is equivalent to giving a bounded t-structure $\scrT$ with heart $\cA$, and a group homomorphism $Z\colon \mathsf{K}_0(\cA)\to \mathbb{C}$, called the central charge, such that for all $0\neq E\in \cA$ the complex number $Z(E)$ lies in the semi-closed upper half-plane
\[
\mathbb{H} \colonequals \{ re^{i\uppi \upvarphi} \mid r>0,\, 0<\upvarphi\leq 1\},
\]
and where furthermore $Z$ must satisfy the Harder--Narasimhan property.
\end{prop}

Write $\Stab\scrT$ for the set of \emph{locally-finite}  stability conditions on $\scrT$.  We do not define these here, as below this condition is automatic for all stability conditions on $\Db(\Lambda_\con)$, since all hearts of bounded t-structures will be equivalent to finite dimensional modules on some finite dimensional algebra.

 \begin{thm}[{\cite[1.2]{B07}}]\label{stab thm}
The space $\Stab\scrT$ has the structure of a complex manifold, and the forgetful map
\[
\Stab\scrT\to\Hom_{\bZ}(\mathsf{K}_0(\scrT),\bC)
\]
is a local isomorphism onto an open subspace of $\Hom_{\bZ}(\mathsf{K}_0(\scrT),\bC)$.
\end{thm}
 Any triangle equivalence $\Upphi\colon \scrT\to \scrT'$ induces a natural map 
\[
\Upphi_*\colon \Stab\scrT\to\Stab\scrT'
\] 
defined by $\Upphi_*(Z,\cA)\colonequals (Z\circ \upphi^{-1},\Upphi(\cA))$, where $\upphi^{-1}$ is the corresponding isomorphism on K-theory $\mathsf{K}_0(\scrT')\simto \mathsf{K}_0(\scrT)$ induced by the functor $\Upphi^{-1}$. In this way,  the group $\Auteq(\scrT)$ of isomorphism classes of autoequivalences of $\scrT$  acts on $\Stab\scrT$.

\subsection{t-structures for \texorpdfstring{$\Db(\Lambda_\con)$}{the derived category}}\label{t-structures cont alg section}
The contraction algebra $\Lambda_\con$ is a \emph{silting-discrete} symmetric algebra \cite[3.3, 4.12]{August1}. Being symmetric, the technical condition of being silting-discrete is equivalent \cite[2.11]{AM} to there being only finitely many basic tilting complexes between $P$ and $P[1]$ (with respect to the silting order $\leq $), for every tilting complex $P$ obtained by iterated irreducible left mutation from the free module $\Lambda_\con$.  Geometrically, for each such $P$, this set is finite since it is in bijection with $R$-schemes obtained by iterated flops of irreducible curves starting from $X$, which is well-known to be finite.

This fact has the following remarkable consequence.

\begin{prop}\label{hearts known}
Suppose that $\scrA$ is the heart of a bounded t-structure on $\Db(\Lambda_\con)$.  Then $\scrA =\scrA_{\upalpha}$  for some $L\in\Mut_0(M)$ and for some $\upalpha\in\Hom_{\dsG}(C_L,C_+)$, where
\begin{align*}
\scrA_\upalpha \colonequals F_{\upalpha}(\mod \uEnd_R(L))
\end{align*}
and $F_{\upalpha}$ is the derived equivalence from Subsection~\ref{def of Falpha} associated to $\upalpha$. 
\end{prop}
\begin{proof}
Since $\Lambda_\con$ is silting-discrete, necessarily $\scrA$ has finite length  \cite{PSZ}.  Furthermore, by the bijections in \cite[\S5]{KY} there exists a silting complex $T$ in $\Db(\Lambda_\con)$ such that, in the notation of \cite[\S5.4]{KY},
\begin{equation}
\scrA=\scrC^{\leq 0}\cap\scrC^{\geq 0}=\{ x\in\Db(\Lambda_\con)
\mid
\Hom_{\Db(\Lambda_\con)}(T,x[i])=0\mbox{ for all }i\neq 0\}.\label{KY heart}
\end{equation}
Since $\Lambda_\con$ is symmetric, silting equals tilting, and so $T$ is a tilting complex.  It is already known (see  \cite[2.12]{August1}) that every tilting complex $T$ in $\Db(\Lambda_\con)$ can be obtained as a composition of mutations from $\Lambda_\con$, so say $T\cong\upmu_\upbeta\Lambda_\con$ for some $\upbeta\in\Hom_{\dsG}(C_+,C_L)$.  Set $\CB\colonequals\uEnd_R(L)$, then \cite[3.10(1)]{August1} gives $F_\upbeta(\upmu_\upbeta\Lambda_\con)\cong \CB$, and hence $F_\upbeta(T)\cong\CB$.

Thus applying $F_\upbeta$ to \eqref{KY heart}, 
\[
F_\upbeta(\scrA)
=\{ y\in\Db(\CB)
\mid
\Hom(\CB,y[i])=0\mbox{ for all }i\neq 0\}
=\mod\CB,
\]
and so applying $F_\upbeta^{-1}=F_{\upbeta^{-1}}$ shows that $\scrA=F_{\upbeta^{-1}}(\mod\CB)$.  Since  $\upbeta^{-1}\in\Hom_{\dsG}(C_L,C_+)$, the result follows. 
\end{proof}

Recall that inside $\Db(\Lambda_\con)$ are the simples $\scrS_1,\hdots,\scrS_n$, which base the K-theory $\mathsf{G}_0(\Lambda_\con)$.  In a similar way, the simple modules $\scrS'_1, \hdots \scrS'_n$ of $\uEnd_R(L)$ base its Grothendieck group.

\begin{cor}\label{arb point}
If $\upsigma\in\Stab\Db(\Lambda_\con)$, then $\upsigma=(Z,\scrA_{\upalpha})$ for some $\upalpha\in\Hom_{\dsG}(C_L,C_+)$ and some $Z$ satisfying $Z(\mathsf{F}_{\upalpha}[\scrS'_i]) \in \mathbb{H}$ for all $i=1,\hdots,n$.
\end{cor}
\begin{proof}
By Proposition~\ref{hearts known} every abelian heart is of the form $\scrA_\upalpha$ for some $\upalpha\in\Hom_{\dsG}(C_L,C_+)$, and hence every point of $\Stab\scrT$ is of the form $(Z, \scrA_{\upalpha})$. To be a stability condition is equivalent to the map $Z\colon \mathsf{K}_0(\scrA_\upalpha)\to\mathbb{C}$ sending all simples of $\scrA_\upalpha$ to $\mathbb{H}$.  Since the simples of $\scrA_\upalpha$ are of the form $F_{\upalpha}(\scrS'_i)$, it follows that $(Z,\scrA_\upalpha)$ is a stability condition precisely when $Z$ satisfies $Z([F_{\upalpha}(\scrS'_i)])=Z(\mathsf{F}_{\upalpha}[\scrS'_i]) \in \mathbb{H}$ for all $i$.
\end{proof}

The general action of $\Auteq\Db(\Lambda_\con)$ on $\Stab\Db(\Lambda_\con)$ simplifies somewhat if we restrict to those standard equivalences given by  $\End_{\dsG}(C_+)$. The functorial assignment $\upalpha \to F_\upalpha$ defines a group homomorphism 
\[
\uppi_1(\scrX)\cong\End_\dsG(C_+) \to\Auteq\Db(\Lambda_\con)
\] 
and we set $\Br$ to be the image of this homomorphism.  Then, using Corollary~\ref{arb point} to describe the points of $\Stab\Db(\Lambda_\con)$, the action of $F_\upbeta\in\Br$ on $\Stab\Db(\Lambda_\con)$ is
\begin{equation}
F_\upbeta\cdot (Z,\scrA_{\upalpha}) 
=
(Z, F_\upbeta(\scrA_{\upalpha}))
=(Z, \scrA_{\upbeta\circ\upalpha}).\label{our action}
\end{equation}
since $\mathsf{F}_\upbeta=\Id$ by Proposition~\ref{trivial K}\eqref{trivial K 2}.

\section{Stability on Contraction Algebras as a Universal Cover}\label{Section 4}

In order to realise $\Stab\Db(\Lambda_\con)$ as a universal cover, fix the isomorphism
\[
\Hom_\mathbb{C}(\mathsf{G}_0(\Lambda_\con),\mathbb{C})
\xrightarrow{\sim}\mathbb{C}^n
\]
given by $Z\mapsto (Z[\scrS_1],\hdots,Z[\scrS_n])=\sum_{i=1}^nZ[\scrS_i]\,\mathbf{e}_i$, where $\scrS_1,\hdots,\scrS_n$ are the simples of $\Lambda_\con$.  Composing this with the forgetful map from Theorem \ref{stab thm}, we thus obtain 
\begin{equation}
p\colon \Stab\Db(\Lambda_\con)\to\Hom_\mathbb{C}(\mathsf{G}_0(\Lambda_\con),\mathbb{C})
\xrightarrow{\sim}\mathbb{C}^n.\label{forgetful}
\end{equation}
Combining with Corollary~\ref{arb point}, $p$ sends an arbitrary point $(Z,\scrA_\upalpha)$ to $(Z[\scrS_1],\hdots,Z[\scrS_n])$.   In this section we will show that $p$ is a regular covering map onto its image.  To do this, it will be convenient to also consider the stability manifolds of the other contraction algebras of $R$, and to track information between them.  

\begin{lem}\label{stab compatible}
For any $\upalpha\in\Hom_\dsG(C_L,C_+)$, the following diagram commutes
\[
\begin{tikzpicture}
\node (A1) at (0,0) {$\Stab\Db(\uEnd_R(L))$};
\node (B1) at (5,0) {$\Stab\Db(\Lambda_\con)$};
\node (A2) at (0,-1.5) {$\Hom(\mathsf{G}_0(\uEnd_R(L)),\mathbb{C})$};
\node (B2) at (5,-1.5) {$\Hom(\mathsf{G}_0(\Lambda_\con),\mathbb{C})$};
\node (A3) at (0,-3) {$\mathbb{C}^n$};
\node (B3) at (5,-3) {$\mathbb{C}^n$};
\draw[->] (A1) --node[above]{$\scriptstyle (F_\upalpha)_*$}(B1);
\draw[->] (A2) --node[above]{$\scriptstyle -\circ\mathsf{F}_\upalpha^{-1}$}(B2);
\draw[->] (A3) --node[above]{$\scriptstyle \upvarphi_L$}(B3);
\draw[bend right,->] (A1) --node[left]{$\scriptstyle $}(A2);
\draw[->] (B1) --node[right]{$\scriptstyle $}(B2);
\draw[->] (A2) -- node[right]{$\scriptstyle \sim$} (A3);
\draw[->] (B2) -- node[right]{$\scriptstyle \sim$} (B3);
\end{tikzpicture}
\]
where the topmost vertical arrows are the forgetful maps, $\mathsf{F}_\upalpha^{-1}$ is the image in K-theory of the inverse of the functor $F_\upalpha$ defined in Subsection~\ref{def of Falpha}, and $\upvarphi_L$ is defined in Subsection~\ref{dual comp subsection}. The right hand vertical composition is $p$.
\end{lem}

\begin{proof}
The top square commutes by definition of $(\mathsf{F}_\upalpha)_*$. For the bottom square, by Proposition~\ref{trivial K}\eqref{trivial K 3} applied to $\upalpha^{-1}$, we have $\mathsf{F}_\upalpha^{-1}=\mathsf{F}_{\upalpha^{-1}}=\mathsf{F}_\upbeta$, where $\upbeta$ is a positive minimal path $C_+ \to C_L$.   Writing $\upbeta=s_{i_t}\circ\hdots\circ s_{i_1}$, then the middle map is the composition
\begin{equation}
(\mathsf{F}_\upbeta)^* \colonequals\quad \mathsf{G}_0(\uEnd_R(L))^*\xrightarrow{\mathsf{F}^*_{i_t}}\hdots\xrightarrow{\mathsf{F}^*_{i_1}}\mathsf{G}_0(\Lambda_\con)^*.\label{chain to hit}
\end{equation}
By Remark~\ref{dual proj basis}, each step is just the tracking of the projectives basing $\mathsf{K}_0(\perf)/[\scrQ_0]$ under the mutation functors. Hence \eqref{chain to hit} is precisely
\begin{equation}
\mathsf{G}_0(\uEnd_R(L))^*\xrightarrow{\upvarphi_{i_t}}\hdots\xrightarrow{\upvarphi_{i_1}}\mathsf{G}_0(\Lambda_\con)^*.\label{chain to hit 2}
\end{equation}
Consider the path $\overline{\upbeta}=s_{i_1}\circ\hdots\circ s_{i_t}\colon C_L\to C_+$. Being the opposite path to $\upbeta$, it follows that $\overline{\upbeta}$ is also positive minimal.  But then by Proposition~\ref{terminal ok} there is a functorial isomorphism
\[
\Upphi_L\colonequals \Upphi_{\overline{\upbeta}}\cong\Upphi_{i_1}\circ\hdots\circ\Upphi_{i_t}.
\]
Hence $\upvarphi_L$, the image of this functor in $\mathsf{K}_0(\perf)/[\scrP_0]$, realises \eqref{chain to hit 2}.
\end{proof}

As is standard, consider the subset of $\mathbb{C}^n$
\[
\mathbb{H}_+ \colonequals \left\{ \sum_{j=1}^n a_j \mathbf{e}_j \mid a_j \in \mathbb{H} \right\} \cong \mathbb{H}^n.  
\]

\begin{cor} \label{effect of p}
For any point of $\Stab\Db(\Lambda_\con)$, which is necessarily of the form $(Z, \scrA_\upalpha)$ for some $\upalpha\in\Hom_\dsG(C_L,C_+)$, 
\[
p(Z, \scrA_\upalpha) = \upvarphi_L \left( \sum_{i=1}^n Z(\mathsf{F}_\upalpha[\scrS_i'])\,\mathbf{e}'_i\right) \in \upvarphi_L( \mathbb{H}_+).
\]
\end{cor}
\begin{proof}
The first statement is Corollary~\ref{arb point}. By definition $\scrA_\upalpha=F_\upalpha(\mod\uEnd_R(L))$, hence 
\begin{align*}
p(Z,\scrA_\upalpha) &= p\bigl( Z, F_\upalpha(\mod\uEnd_R(L))\bigr)
\\
&=
p\circ (F_\upalpha)_*( Z \circ \mathsf{F}_\upalpha, \mod \uEnd_R(L))\tag{since $\mathsf{F}_{\upalpha}\circ \mathsf{F}_{\upalpha}^{-1}=\Id$}\\
&= \upvarphi_L\left( \sum_{i=1}^n Z(\mathsf{F}_\upalpha[\scrS_i'])\,\mathbf{e}'_i\right).\tag{by Lemma~\ref{stab compatible}} 
\end{align*}
The final statement that $p(Z,\scrA_\upalpha)\in\upvarphi_L(\bH_+)$ again follows from Corollary~\ref{arb point}.
\end{proof}

\subsection{A Covering Map}
As in Section \ref{t-structures cont alg section}, let $\Br$ be the image of the homomorphism
\[
\uppi_1(\scrX)\to\Auteq\Db(\Lambda_\con)
\]
sending $\upbeta\mapsto F_\upbeta$. Then $F_{\upbeta}\in\Br$ acts on the space of stability conditions via the action in \eqref{our action}. In this subsection we will  establish that this action is free and properly discontinuous, so that $\Stab\Db(\Lambda_\con) \to \Stab\Db(\Lambda_\con)/\Br$ is a covering map. 

The following is one of our main technical results.  It establishes a condition under which elements of $\Br$, inside the autoequivalence group of $\Db(\Lambda_\con)$, are the identity.  The proof is via Fourier--Mukai techniques.  Forgetting the ambient geometry is thus a bad idea: it seems extremely difficult to establish the following result in a purely algebraic manner. 
\begin{thm}\label{functor is identity}
Suppose that $\upalpha\in\End_{\dsG}(C_+)$ satisfies $F_\upalpha(\Lambda_\con) \cong \Lambda_\con$. Then there is a functorial isomorphism $F_\upalpha \cong \Id$. 
\end{thm}
\begin{proof}
By the assumption, the standard equivalence $F_\upalpha$ is induced by the one-sided tilting complex $\Lambda_\con$. By the usual lifting argument (see e.g.\ \cite[2.3]{RZ}), the bimodule complex defining $F_\upalpha$ must be isomorphic to ${}_{1}({\Lambda_\con})_{\upzeta}$ as bimodules, for some algebra automorphism $\upzeta\colon\Lambda_\con\to\Lambda_\con$.  Hence $F_\upalpha$ is induced by this algebra automorphism.

Since $F_\upalpha$ induces a Morita equivalence, it must take  simples to simples.  Furthermore, as $F_\upalpha$ is the  identity on K-theory $\mathsf{G}_0(\Lambda_\con)$ by Proposition~\ref{trivial K}, $F_\upalpha$ must fix all simples. 

Now consider the commutative diagram
\[
\begin{tikzpicture}
\node (Acon) at (0,0) {$\Db(\Lambda_\con)$};
\node (A) at (4,0) {$\Db(\Lambda)$};
\node (X) at (8,0) {$\Db(\coh X)$};
\node (Bcon) at (0,-2) {$\Db(\Lambda_\con)$};
\node (B) at (4,-2) {$\Db(\Lambda)$};
\node (Y) at (8,-2) {$\Db(\coh X)$};
\draw[->] (Acon) --
node[above] { $\scriptstyle \res$} (A);
\draw[->] (A) --
node[above] { $\scriptstyle - \otimes^{\bf L}_{\Lambda}\scrV$} node [below] {$\scriptstyle\sim$} (X);
\draw[->] (Bcon) --
 node[above] { $\scriptstyle \res$} (B);
\draw[->] (B) --
 node[above] { $\scriptstyle - \otimes^{\bf L}_{\Lambda} \scrV$} node [below] {$\scriptstyle\sim$} (Y);
 
\draw[->] (Acon) --node [left] {$\scriptstyle F_{\upalpha} $} node [right] {$\scriptstyle\sim$} (Bcon);
\draw[->] (A) --node [left] {$\scriptstyle \Upphi_{\upalpha}$} node [right] {$\scriptstyle\sim$} (B);
\draw[->] (X) -- node[left] {$ \scriptstyle  \Flop_\upalpha$} node [right] {$\scriptstyle\sim$} (Y);
\end{tikzpicture}
\]
where the left hand side is \eqref{commutes for alpha}, and the right hand side can be obtained by iterating the right hand side of Theorem~\ref{intertwinement}.   Since $F_\upalpha$ fixes simples, the left hand commutative diagram  implies that $\Upphi_\upalpha$ fixes the simples $\scrS_1,\hdots,\scrS_n$.  By \cite[3.5.8]{VdB}, across the right hand commutative diagram, this in turn implies that $\Flop_\upalpha$ fixes the sheaves $\scrO_{\mathrm{C}_1}(-1),\hdots,\scrO_{\mathrm{C}_n}(-1)$, where each $\mathrm{C}_i\cong\mathbb{P}^1$.  
Since the flop functor and its inverse both map $\scrO_X$ to $\scrO_X$ \cite[(4.4)]{Bridgeland} and commute with the pushdown to $\Spec R$ (see e.g.\ \cite[7.16]{DW1}), by the standard Fourier--Mukai argument (see \cite[4.3]{HomMMP}, which itself is based on \cite{T08}), $\Flop_\upalpha\cong h_*$ for some isomorphism $h \colon X\to X$ that commutes with the pushdown.  But this isomorphism is the identity on the dense open set obtained by removing the exceptional locus, and hence it must be the identity.

It follows that $\Flop_\upalpha\cong\Id$ and hence $\Upphi_\upalpha\cong\Id$.  Restricting the left hand commutative diagram to $\mod \Lambda_\con$, we obtain a commutative diagram
\[
\begin{tikzpicture}
\node (A) at (0,0) {$\mod\Lambda_\con$};
\node (X) at (4,0) {$\mod_{\{1,\hdots,n\}}\Lambda$};
\node (B) at (0,-2) {$\mod\Lambda_\con$};
\node (Y) at (4,-2) {$\mod_{\{1,\hdots,n\}}\Lambda$};
\draw[->] (A) --
node[above] { $\scriptstyle \res$} node [below] {$\scriptstyle\sim$} (X);
\draw[->] (B) --
 node[above] { $\scriptstyle \res$} node [below] {$\scriptstyle\sim$} (Y);
\draw[->] (A) --node [left] {$\scriptstyle F_\upalpha$} (B);
\draw[->] (X) -- node[right] {$ \scriptstyle \Upphi_\upalpha\cong\Id$} (Y);
\end{tikzpicture}
\]
where $\mod_{\{1,\hdots,n\}}\Lambda$ denotes those $\Lambda$-modules with a finite filtration by the simples $\scrS_1,\hdots,\scrS_n$.  In particular, $F_\upalpha$ restricted to $\mod\Lambda_\con$ is functorially isomorphic to the identity.  Hence $\upzeta$ is an inner automorphism (see e.g.\ \cite[2.8.16]{linckelmann}), which in turn implies that $F_\upalpha\cong\Id$.
\end{proof}

\begin{cor}\label{free and prop discont}
For each $x\in \Stab\Db(\Lambda_\con)$ there is an open neighbourhood $\scrU$ of $x$ such that $\scrU\cap F_{\upbeta}(\scrU)=\emptyset$ for all $1\neq F_\upbeta\in \Br$.
\end{cor}
\begin{proof}
Consider the open neighbourhood $\scrU$ of $x$ defined by 
\[
\scrU\colonequals\{y\in\Stab{}\Db(\Lambda_\con) \mid d(x,y)<1/4\},
\] 
where $d(-,-)$ is the metric on stability conditions introduced in \cite[Section 6]{B07}.  Suppose that $\scrU\cap F_\upbeta(\scrU)\neq\emptyset$ for some $F_\upbeta\in\Br$.  We will show that $F_\upbeta\cong\Id$. 

As the two open balls of radius $\frac{1}{4}$ intersect, every $y\in  \scrU$ must satisfy  $d\bigl(y,(F_\upbeta)_*y\bigr)<1$.  Furthermore, the central charges of $y$ and $(F_\upbeta)_*y$ are equal by Proposition \ref{trivial K}\eqref{trivial K 2} and the top commutative diagram in Lemma~\ref{stab compatible}. 
Using \cite[Lemma 6.4]{B07} it follows immediately that  $y=(F_\upbeta)_*y$, for every $y\in \scrU$.

In particular, by Corollary~\ref{arb point}, say $x=(Z,\scrA_\upalpha)$ for some $\upalpha\in\Hom_{\dsG}(C_L,C_+)$.  Since $x\in \scrU$,  the property $(F_\upbeta)_*(x)=x$ implies that $\scrA_{\upbeta\circ\upalpha}=\scrA_\upalpha$, and so $F_\upbeta$ restricts to an equivalence $\scrA_\upalpha\to\scrA_{\upalpha}$.  In turn, this implies that the composition
\[
F_{\upalpha^{-1}\upbeta\upalpha}=F_\upalpha^{-1}\circ F_\upbeta\circ F_\upalpha\colon\Db(\Gamma_\con)\to\Db(\Gamma_\con),
\] 
where $\Gamma_\con\colonequals \uEnd_R(L)$, restricts to an equivalence $\mod\Gamma_\con\to\mod\Gamma_\con$.  It follows that $F_{\upalpha^{-1}\upbeta\upalpha}(\Gamma_\con)$ must then be a basic tilting module, given that $\Gamma_\con$ is.   Since $\Gamma_\con$ is symmetric, the only such module is $\Gamma_\con$, and so $F_{\upalpha^{-1}\upbeta\upalpha}(\Gamma_\con)\cong\Gamma_\con$. By Proposition~\ref{functor is identity} applied to the contraction algebra $\Gamma_\con$, we conclude that $F_{\upalpha^{-1}\upbeta\upalpha}\cong\Id$, and hence $F_\upbeta\cong\Id$.
\end{proof}

\begin{cor}\label{covering map}
The map $\Stab\Db(\Lambda_\con) \to \Stab\Db(\Lambda_\con)/\mathsf{PBr}$ is the universal covering map, with Galois group $\mathsf{PBr}$. 
\end{cor}
\begin{proof}
$\Stab\Db(\Lambda_\con)$ is contractible since contraction algebras are silting-discrete \cite[4.12]{August1}, and silting-discrete algebras have contractible stability manifolds \cite{PSZ}. In particular, $\Stab\Db(\Lambda_\con)$ is path connected and so the given map is a regular covering map by Corollary~\ref{free and prop discont} together with the standard \cite[1.40(a)(b)]{Hat}. The covering is clearly universal, since $\Stab\Db(\Lambda_\con)$ is contractible and hence simply connected. 
\end{proof}

\subsection{The Regular Cover to the Complexified Complement}
In this subsection we will establish that  $p$ induces an isomorphism 
\begin{equation}
\Stab\Db(\Lambda_\con)/\Br \xrightarrow{\sim} \mathbb{C}^n \backslash \scrH_{\mathbb{C}}.\label{want iso}
\end{equation}
Combining with Corollary~\ref{covering map} will then prove that \eqref{forgetful} is the universal covering map onto its image, with Galois group $\Br$.  We will establish \eqref{want iso} in two steps: first by showing that $p$ has image $\mathbb{C}^n\backslash \scrH_{\mathbb{C}}$, then second by establishing that \eqref{want iso} is well-defined and injective.

Our proof makes use of the following key combinatorial result, which is folklore when $\scrH$ is an ADE root system.  In our mildly more general setting here, the proof, which is basically the same, can be found in \cite[Appendix A]{HW2}.
\begin{prop}\label{chamber decomposition} 
With notation as above, the following hold. \begin{enumerate}
\item\label{chamber decomposition 1}   If $\upalpha$ and $\upbeta$ terminate at $C_+$, then 
\[
\upvarphi_{s(\upalpha)}(\mathbb{H}_+)\cap\upvarphi_{s(\upbeta)}(\mathbb{H}_+) \neq \emptyset \iff s(\upalpha)=s(\upbeta).
\]
\item\label{chamber decomposition 2}   There is a disjoint union
\[
\mathbb{C}^n \backslash \scrH_{\mathbb{C}} = \bigsqcup_{L \in \Mut_0(N)} \upvarphi_L( \mathbb{H}_+).
\]
\end{enumerate}
\end{prop}

The above combinatorics lead directly to the following.  In the special case when $\scrH$ is an ADE root system, the following result is already implicit in \cite{B3} and \cite{T08}. 
\begin{cor} \label{image of p}
The image of $p$ is $\mathbb{C}^n \backslash \scrH_{\mathbb{C}}$.
\end{cor}
\begin{proof}
By Proposition~\ref{chamber decomposition}\eqref{chamber decomposition 2}, all the sets $ \upvarphi_L( \mathbb{H}_+)$ avoid the complexified hyperplanes, so by Corollary~\ref{effect of p}  the image of $p$ lies in $\mathbb{C}^n \backslash \scrH_{\mathbb{C}}$. Further, given any $z \in \mathbb{C}^n \backslash \scrH_{\mathbb{C}}$, we may write $z=\upvarphi_L(h)$ for some $L \in \Mut_0(N)$ and some $h\in\bH_+$, again by Proposition~\ref{chamber decomposition}\eqref{chamber decomposition 2}.  Since the standard heart in $\Db(\uEnd_R(L))$ has finite length, we can find a stability condition $\upsigma\in\Stab \Db(\uEnd_R(L) )$ which maps to $h$ via the left hand vertical composition in Lemma~\ref{stab compatible}. Then for any $\upalpha \in \Hom_\dsG( C_L, C_+)$, the commutative diagram in Lemma~\ref{stab compatible} shows that $(F_\upalpha)_*(\upsigma)\in\Stab \Db(\Lambda_\con )$ maps, under $p$, to $z$.
\end{proof}

The following shows that \eqref{want iso} is both well-defined, and injective.

\begin{lem}\label{disjoint pieces}
For $\upsigma_1,\upsigma_2\in\Stab\Db(\Lambda_\con)$, then
\[
p(\upsigma_1)=p(\upsigma_2)\iff
\upsigma_1=(F_\upgamma)_*\upsigma_2\mbox{ for some } F_\upgamma\in\Br.
\]
\end{lem}
\begin{proof}
Note that ($\Leftarrow$) is clear since, by Proposition \ref{trivial K}\eqref{trivial K 2}, the action in \eqref{our action} of a pure braid does not effect the central charge of a stability condition. 

For ($\Rightarrow$), recall that by Corollary~\ref{arb point} we can assume that $\upsigma_1=(Z_1,\scrA_{\upalpha})$ and $\upsigma_2=(Z_2,\scrA_{\upbeta})$ where $\upalpha \in \Hom_\dsG(C_{L_1}, C_+)$ and $\upbeta \in \Hom_\dsG(C_{L_2},C_+)$. If $p(\upsigma_1)=p(\upsigma_2)$, then certainly $Z_1 =Z_2$ since $p$ is simply the forgetful map followed by an isomorphism.  

Furthermore, by Corollary~\ref{effect of p}, we see that $\upvarphi_{s(\upalpha)}(\mathbb{H}_+)\cap\upvarphi_{s(\upbeta)}(\mathbb{H}_+)\neq \emptyset$, since the intersection contains $p(\upsigma_1)=p(\upsigma_2)$.  Hence by Proposition~\ref{chamber decomposition}\eqref{chamber decomposition 1} it follows that $s(\upalpha)=s(\upbeta)$ and thus we can consider the composition $\upgamma = \upbeta \circ \upalpha^{-1} \in \End_\dsG( C_+)$. Then $F_\upgamma\in\Br$ and 
\begin{align*}
(F_\upgamma)_*(Z_1,\scrA_{\upalpha}) 
&=  (Z_1,\scrA_{\upgamma\circ \upalpha})\tag{by \eqref{our action} since $F_\upgamma\in\Br$}\\
&=(Z_1,\scrA_{\upbeta \circ \upalpha^{-1} \circ \upalpha})\tag{using $\upgamma = \upbeta \circ \upalpha^{-1}$ }\\
&=(Z_2,\scrA_{\upbeta}),\tag{since $Z_1=Z_2$ and $\upalpha^{-1} \circ \upalpha=\Id$} 
\end{align*}
proving the statement.
\end{proof}

\begin{cor} \label{homeomorphism}
The map $p \colon \Stab\Db(\Lambda_\con) \to \mathbb{C}^n \backslash \scrH_{\mathbb{C}}$ induces a homeomorphism
\[
\Stab\Db(\Lambda_\con)/\Br \to \mathbb{C}^n \backslash \scrH_{\mathbb{C}}.
\]
\end{cor}
\begin{proof}
The map $p$ is surjective by Corollary~\ref{image of p}, and by Lemma~\ref{disjoint pieces} it induces the bijection in the statement.  The induced map is itself  a homeomorphism by the definition of the quotient topology, and the fact that $p$ is a local homeomorphism.
\end{proof}

The following is our main result.

\begin{cor}\label{contractible alg}
The map $p\colon\Stab\Db(\Lambda_\con)\to \mathbb{C}^n \backslash \scrH_{\mathbb{C}}$ is the universal cover, with Galois group $\Br$.  Furthermore, $\Stab\Db(\Lambda_\con)$ is contractible.
\end{cor}
\begin{proof}
The first statement is obtained by composing the universal cover from Corollary~\ref{covering map} with the homeomorphism from Corollary~\ref{homeomorphism}. As already stated, the second part follows from \cite{PSZ} and \cite[4.12]{August1}.
\end{proof}

\section{Corollaries}\label{Section 5}

In this section we prove the five main corollaries stated in the introduction.  For an ADE root system, it is well-known \cite{Brieskorn} that $\uppi_1(\mathfrak{h}_{\mathrm{reg}}/W)$ is isomorphic to the associated ADE braid group.  Recall that the $K(\uppi,1)$-conjecture for ADE braid groups, which is already a theorem in this setting, asserts that the universal cover of $\mathfrak{h}_{\mathrm{reg}}/W$ is contractible.  

\begin{cor}\label{kpi1main}
The $K(\uppi,1)$-conjecture holds for all ADE braid groups.
\end{cor}
\begin{proof}
As in \cite[\S3]{T08} or \cite[\S4.3]{KM}, we may choose a flopping contraction for which the hyperplane arrangement $\scrH$ is an ADE root system $\mathfrak{h}$.  It is well known that the complexified complement $\scrX=\mathfrak{h}_{\mathrm{reg}}$.  Write $W$ for the Weyl group, which is finite, thus clearly the covering map $ \mathfrak{h}_{\mathrm{reg}}\to \mathfrak{h}_{\mathrm{reg}}/W$ has finite fibres. It follows that the composition 
\[
\Stab\Db(\Lambda_\con)\to \scrX=\mathfrak{h}_{\mathrm{reg}}\to \mathfrak{h}_{\mathrm{reg}}/W
\]
is then also a covering map.  As above, it is well-known that $\uppi_1(\mathfrak{h}_{\mathrm{reg}}/W)$ is the braid group. By Corollary~\ref{contractible alg}, the fact that $\Stab\Db(\Lambda_\con)$ is contractible implies that it is simply connected.  Hence the composition is the universal cover, and furthermore the universal cover is contractible.
\end{proof}

\begin{cor}\label{faith cont alg}
The homomorphism $\uppi_1(\scrX)\to\Auteq\Db(\Lambda_\con)$ sending $\upalpha\mapsto F_\upalpha$ is injective.
\end{cor}
\begin{proof}
As in Sections~\ref{Section 3} and \ref{Section 4}, by definition $\Br$ is the image of this homomorphism.  Since $p$ is a regular covering map, as is standard \cite[1.40(c)]{Hat} there is a short exact sequence of groups
\begin{equation}
1\to \uppi_1(\Stab\Db(\Lambda_\con)) \to \uppi_1(\scrX) \to \Br\to 1\label{Hat ses}
\end{equation}
where $\uppi_1(\scrX) \to \Br$ as before takes $\upalpha\mapsto F_\upalpha$. However, since by Corollary~\ref{contractible alg} $\Stab\Db(\Lambda_\con)$ is contractible, its fundamental group is trivial.
\end{proof}

For $\upalpha\in\Hom_{\dsG}(C_L,C_+)$, set $T_\upalpha\colonequals F_{\upalpha}(\uEnd_R(L))$, which is necessarily a tilting complex for $\Lambda_\con$ since equivalences map tilting complexes to tilting complexes.
\begin{cor}\label{cont 1-1}
The map $\upalpha \mapsto T_\upalpha$ is a bijection from morphisms in the Deligne groupoid ending at $C_+$ to the set of basic one-sided tilting complexes of $\Lambda_{\con}$, up to isomorphism.
\end{cor}
\begin{proof}
Surjectivity is already known. Indeed, since $\Lambda_\con$ is silting-discrete, by \cite[3.16(2)]{August1} every standard derived equivalence from $\Db(\Lambda_{\con})$ is, up to algebra automorphism, isomorphic to $F_\upbeta$ for some $\upbeta\in \Hom_{\dsG}(C_+,C_L)$.  In particular every one-sided tilting complex for $\Lambda_{\con}$ is isomorphic to $F_{\upbeta}^{-1}(\uEnd_R(L))$ for some $\upbeta \in\Hom_{\dsG}(C_+,C_L)$. Since $\upbeta^{-1} \in\Hom_{\dsG}(C_L,C_+)$ and $F_{\upbeta}^{-1}=F_{\upbeta^{-1}}$, surjectivity of the map follows. 

The content is that the map is also injective.  Suppose that $\upalpha\in\Hom_{\dsG}(C_L,C_+)$ and $\upbeta \in\Hom_{\dsG}(C_M,C_+)$ are such that $T_\upalpha \cong T_\upbeta$, where $L,M \in \Mut_0(N)$. Then, by definition, 
\begin{align}
F_{\upalpha}(\uEnd_R(L)) \cong F_{\upbeta}(\uEnd_R(M)) \label{iso tilting}
\end{align}
and hence $\scrA_\upalpha = \scrA_\upbeta$. In particular, choosing any central charge $Z$, we have $(Z, \scrA_\upalpha) = (Z, \scrA_\upbeta)$ and hence $p(Z, \scrA_\upalpha)=p(Z, \scrA_\upbeta)$. But Corollary~\ref{effect of p} then implies that $\upvarphi_L(\mathbb{H}_+)$ and $\upvarphi_M(\mathbb{H}_+)$ intersect, which by Proposition~\ref{chamber decomposition} implies that $L \cong M$.

Set $\CB\colonequals\uEnd_R(L)$.  Applying $F_{\upbeta}^{-1}=F_{\upbeta^{-1}}$ to \eqref{iso tilting} gives $F_{\upbeta^{-1} \circ \upalpha}(\CB) \cong \CB$.  Thus by applying Theorem~\ref{functor is identity} to the composition $\upbeta^{-1} \circ \upalpha \in \End_\dsG( C_L)$ we deduce that there is an isomorphism $F_{\upbeta^{-1} \circ \upalpha} \cong \Id$. Corollary~\ref{faith cont alg} applied to the contraction algebra $\CB$ then shows that $\upbeta^{-1} \circ \upalpha$ must be the identity, and hence $\upalpha =\upbeta$ in the Deligne groupoid. 
\end{proof}
\begin{rem}
If we instead assign to a path $\upalpha \colon C_+ \to C_L$ the tilting complex $F_{\upalpha}^{-1}(\uEnd_R(L))$,  we equivalently obtain a bijection between the paths in the Deligne groupoid that \emph{start} at $C_+$ and basic one-sided tilting complexes of $\Lambda_{\con}$.
\end{rem}

\begin{cor}\label{faithful geo cor}
The homomorphism $\uppi_1(\scrX)\to\Auteq\Db(\coh X)$ sending $\upalpha\mapsto \Flop_\upalpha$ is injective.
\end{cor}
\begin{proof}
Suppose that $\upalpha$ belongs to the kernel, so $\Flop_\upalpha=\Id$. Since $\Upphi_\upalpha$ is functorially isomorphic to $\Flop_\upalpha$ after tilting by Theorem~\ref{intertwinement}, necessarily  $\Upphi_\upalpha\cong\Id$.  The left-hand part of the commutative digram in Theorem~\ref{intertwinement}  then implies that  $F_\upalpha(\Lambda_\con)$ maps, under restriction of scalars, to $\Lambda_\con$.  It follows that $F_\upalpha(\Lambda_\con)\cong\Lambda_\con$ in $\Db(\Lambda_\con)$, see e.g.\ \cite[6.6]{August2}. By Lemma~\ref{functor is identity}, there is a functorial isomorphism $F_\upalpha\cong\Id$.  By Corollary~\ref{faith cont alg} we see that $\upalpha=1$, proving the statement.
\end{proof}

\begin{cor}\label{StabC contractible}
With the notation as in the introduction, $\cStab{}\scrC$ is contractible.
\end{cor}
\begin{proof}
By \cite{HW2} there is a regular covering map $\cStab{}\scrC\to\scrX$ with Galois group $G$ equal to the image of $\uppi_1(\scrX)\to\Auteq\Db(\coh X)$.  But by Corollary~\ref{faithful geo cor}, the map $\uppi_1(\scrX)\to G$ is an isomorphism.  By the corresponding version of \eqref{Hat ses}, $\uppi_1(\cStab{}\scrC)$ is trivial and so the cover is universal.  Universal covers are unique, so by Corollary~\ref{contractible alg} it follows that $\cStab{}\scrC$ is contractible.
\end{proof}

\end{document}